\documentclass{amsart}

\usepackage[dvips]{color}
\usepackage{graphicx}
\usepackage{epstopdf}
\usepackage{paralist}
\usepackage{latexsym}
\usepackage{amssymb}
\usepackage{amsmath}
\usepackage{amsthm}
\usepackage{amscd}
\usepackage{mathrsfs}
\usepackage[utf8]{inputenc} 
\usepackage[T1]{fontenc}
\usepackage[colorlinks=true]{hyperref}
\hypersetup{urlcolor=blue, citecolor=red}

\theoremstyle{plain}
\newtheorem*{thmA}{Theorem A}
\newtheorem*{thmB}{Theorem B}
\newtheorem{thm}{Theorem}[section]
\newtheorem{lem}[thm]{Lemma}
\newtheorem{prop}[thm]{Proposition}
\newtheorem{cor}[thm]{Corollary}

\theoremstyle{definition}

\newtheorem{rem}[thm]{Remark}

\theoremstyle{remark}

\newcommand{\LS}{\ensuremath{\underset{N=1}{\overset{\infty}{\cap}} \, {\underset{i=N}{\overset{\infty}{\cup}}}\,}}

\author[J.\ Chaika]{Jon Chaika}\email{chaika@math.utah.edu}\address{Department of Mathematics, University of Utah, 155 S 1400 E, Room 233, Salt Lake City, UT~84112, USA}
\author[D. Constantine]{David Constantine}\email{dconstantine@wesleyan.edu}\address{Department of Mathematics and Computer Science, Wesleyan University, 265 Church Street, Middletown, CT~06459, USA}



\thanks{The first author is supported by NSF grants DMS-1004372, 135500, 1452762, the Sloan Foundation, a Warnock chair, and a Poincar\'e chair.}

\begin{document}

\title[Quantitative shrinking targets]{A Quantitative Shrinking Target result on Sturmian sequences for rotations}

\maketitle

% Enter the first author's name and address:
%\centerline{\scshape J. Chaika}
%\medskip
%{\footnotesize
% please put the address of the first author
% \centerline{Department of Mathematics, University of Utah}
%   \centerline{155 S 1400 E, Room 233}
%   \centerline{Salt Lake City, UT~84112, USA}
%} % Do not forget to end the {\footnotesize by the sign }

%\medskip

%\centerline{\scshape D. Constantine}
%\medskip
%{\footnotesize
 % please put the address of the second  and third author
% \centerline{ Department of Mathematics and Computer Science, Wesleyan University}
%   \centerline{265 Church Street}
%   \centerline{Middletown, CT~06459, USA}
%}

%\bigskip

% The name of the associate editor will be entered by an editorial staff
% "Communicated by the associate editor name" is not needed for special issue.
% \centerline{(Communicated by the associate editor name)}

\begin{abstract}
Let $R_\alpha$ be an irrational rotation of the circle, and code the orbit of any point $x$ by whether $R_\alpha^i(x)$ belongs to $[0,\alpha)$ or $[\alpha, 1)$ -- this produces a Sturmian sequence. A point is undetermined at step $j$ if its coding up to time $j$ does not determine its coding at time $j+1$. We prove a pair of results on the asymptotic frequency of a point being undetermined, for full measure sets of $\alpha$ and $x$.
\end{abstract}

%
%%
%%%
%%%%INTRODUCTION

\section{Introduction}
 
 \subsection{Statement of the problem and main results}

In this paper we study a shrinking target problem. Let $\alpha \in [0,1)$, let $R_\alpha:[0,1)\to [0,1)$ be the rotation $R_\alpha(x) = x+\alpha \mbox{ (mod 1)}$, and let $\lambda$ denote Lebesgue measure. The following theorem, due to Weyl, is well known.

\begin{thm}
Let $\alpha \notin \mathbb{Q}$. Then for any $x,y \in [0,1)$, and any $\epsilon>0$,
\[ \lim_{N\to \infty} \frac{\sum_{i=1}^N \chi_{B(y,\epsilon)}(R^i_\alpha x)}{\sum_{i=1}^N \lambda(B(y,\epsilon))} = 1. \]
\end{thm}

\noindent That is, the asymptotic for the number of visits of the orbit of $x$ to the target set $B(y,\epsilon)$ by step $N$ is given by the sum of the size of the target over those $N$ steps.

The statement is written here in a slightly unusual way -- the denominator is clearly $N2\epsilon$ (assuming $\epsilon\leq\frac{1}{2}$ and identifying $[0,1)$ with $S^1$). But it suggests the following sort of problem. Let $\{B_i\}$ be a sequence of measurable sets in $[0,1)$. What can be said about the behavior of $\sum_{i=1}^N\chi_{B_i}(R_\alpha^i x)$; in particular, is it asymptotic to $\sum_{i=1}^N\lambda(B_i)$?

This is, of course, an enormously varied problem. Cases which have generated significant interest are \emph{shrinking target problems}, in which the $B_i$ form a decreasing, nested chain. By the Borel-Cantelli Lemma, the cases of real interest are when $\sum_{i=1}^\infty \lambda(B_i) = \infty.$ Several results on this problem for rotations and for interval exchange transformations are contained in \cite{chaika_constantine}, including the following.

\begin{thm}\cite{chaika_constantine}
For all $\alpha$ satisfying an explicit, full measure diophantine condition, and for any sequence $\{r_i\}$ such that $ir_i$ is non-increasing and $\sum_{i=1}^\infty r_i=\infty$, and any $y$
\[ \lim_{N\to \infty} \frac{\sum_{i=1}^N \chi_{B(y,r_i)}(R^i_\alpha x)}{\sum_{i=1}^N 2r_i} =1 \]
for almost every $x$.
\end{thm}

In this paper we consider another shrinking target problem for rotations, but one whose targets arise in a very different way.  Rather than being subject to some pre-determined analytic constraint (as for the sequence $\{r_i\}$) the targets arise from the dynamics of the rotation itself.

Let $\alpha$ be given. Let $\mathcal{P}=\{A_0, A_1\}$ be the partition of $[0,1)$ given by $A_0=[0,\alpha), A_1=[\alpha, 1)$.  The bi-infinite sequences $(c_i(x))_{i\in\mathbb{Z}}$ defined by $c_i(x)=j$ if $R_\alpha^ix \in A_j$ are known as Sturmian sequences (see, e.g. \cite{Fogg}, Ch. 6). These are sequences of minimal complexity, or with minimal block growth. They were introduced by Hedlund and Morse \cite{HM}, and have been studied extensively.

 For a sequence $(c_0, c_1, \ldots)$ (finite or infinite) of 0's and 1's, let $C_{c_0,c_1, \ldots}=\{x: R_\alpha^ix\in A_{c_i} \mbox{ for all } i\}$. If $x\in C_{c_0, c_1 \ldots}$, then $(c_0, c_1, \ldots)$ is a \emph{coding} for the orbit of $x$ (or a portion thereof, if the sequence is finite). Let $\Sigma$ be the set of finite codings $c_0, c_1, \ldots, c_n$ which actually occur, i.e. for which $C_{c_0, \ldots, c_n}\neq \emptyset$.  Let 
\[V_j = \{x : x\in C_{c_0, \ldots, c_j} \mbox{ and such that } c_0, \ldots, c_j, 0 \mbox{ and } c_0, \ldots, c_j, 1 \in \Sigma\}.\] 
This is the set of \emph{`undetermined' points at step} $j$, that is, points whose coding up to step $j$ does not determine the coding at step $j+1$.  The word $c_0,\ldots,c_j$ is also known as a \emph{right special word} (see, e.g., \cite[\S2.1.1]{lothaire}.)

We want to find asymptotics on how often a point is undetermined; specifically, we will prove

\begin{thmA}\label{undetermined thm}
For almost all $\alpha$,
\[ \lim _{n\to \infty} \frac{\log \sum_{j=1}^{n} \chi_{V_j}(x)}{\log \sum_{j=1}^n \lambda(V_j)} = 1 \]
for almost all $x\in (0,1)$.
\end{thmA}

As in \cite{chaika_constantine}, the full measure condition on $\alpha$ is a diophantine condition involving the continued fraction expansion of $\alpha$. It will be stated explicitly in the proof.

To understand why Theorem A constitutes a shrinking target problem, consider the following.  Let $\mathcal{P}_j = \vee_{k=0}^j R_\alpha^k\mathcal{P}$, the partition generated by $\mathcal{P}$ and its first $j$ translates.  For $x\in X$, denote by $[[x]]_j$ the atom of $x$ in $\mathcal{P}_j$.  The coding $c_0, \ldots, c_j$ determines only the atom $[[R_\alpha^jx]]_j$.  A point $x$ will belong to $V_j$ if and only if $R_\alpha^jx$ is in $[[1-\alpha]]_j$ as the image of this atom under one more rotation contains points in both $A_0$ and $A_1$.
We will denote $[[1-\alpha]]_j$ by $U_j$ -- these are the shrinking targets which we are trying to hit.  Note that $U_j = R_\alpha^j(V_j)$.

The logarithms in Theorem A indicate a weaker asymptotic result than in \cite{chaika_constantine}. The stronger version is not true:

\begin{thmB}\label{no stronger}
For almost all $\alpha$,
\[ \lim_{N\to \infty} \frac{\sum_{j=1}^N \chi_{V_j}(x)}{\sum_{j=1}^N \lambda(V_j)}\]
does not exist for almost every $x\in[0,1)$.
\end{thmB}

\noindent Thus, Theorem A is in some sense the best one can hope for in this setting, an interesting contrast with the stronger results obtained for targets of the form $B(y,r_i)$.

%
%%
%%%
%%%%%%%%%%%OUTLINE

\subsection{Notation and an outline of the paper}

The key tool throughout the paper is the continued fraction expansion of $\alpha$ and its close relationship to the dynamics of the rotation by $\alpha$. Throughout, $\alpha\in (0,1)$ is assumed to be irrational. We write 
\[ \alpha  = [0; a_1, a_2, a_3, \ldots]\]
for the continued fraction expansion of $\alpha$. Note that the \emph{elements} $a_i$ of the continued fraction depend on $\alpha$; we will at times write $a_i(\alpha)$ to emphasize this dependence. The \emph{convergents} to $\alpha$ are the rationals $\frac{p_k}{q_k}$. The $k^{th}$ convergent is the best rational approximation to $\alpha$ with denominator $\leq q_k$. The $q_k$ can be computed by the recurrence relation $q_{k+1}=a_{k+1}q_k+q_{k-1}; q_0=1, q_1=a_1$.

We will prove Theorem B first, in Section \ref{sec:failure}. The almost sure existence of elements of the continued fraction expansion which are very large in relation to the preceding elements drives the argument.

Theorem A is proved in Section \ref{sec:main proof}. There we prove a set of looser bounds on $\sum_{i=1}^n a_i$ and $\sum_{j=1}^{q_n}\chi_{V_j}(x)$ which hold for almost all $\alpha$ and which are sufficient for the statement of Theorem A.

%
%%
%%%
%%%%%%%%%%%%%%FAILURE OF STRONGER RESULT

\section{Failure of a stronger convergence}\label{sec:failure}

%Before turning to the proof of Theorem A, we give an argument as to why there is no stronger theorem along the lines of convergence of
%\begin{equation} \label{quotient} 
%	\frac{\sum_{j=1}^{n} \chi_{V_j}(x)}{\sum_{j=1}^n \lambda(V_j)}.
%\end{equation}  

We start with the proof of Theorem B. First, we prove the almost-sure existence of very large elements $a_n$ for the continued fraction expansion. We then use this to show that, for very long stretches of time certain points are undetermined more often than $\sum_{j=1}^n \lambda(V_j)$ predicts. 

%%%%%%%%%%%%%
%%%%%%%%%%%%%

\begin{prop}\label{big time} 
For any $C \in \mathbb{R}$ and almost every $\alpha$ there exist infinitely many $m$ such that 
\begin{equation}\label{eqn:a condition}
	a_m>C\underset{i=1}{\overset{m-1}{\sum}} a_i.
\end{equation}
\end{prop}

We need a series of preliminary results to prove this. The following lemma appears in \cite[page 60]{khinchin}.

\begin{lem}\label{bounds} 
For any $n, b_1,...,b_n \in \mathbb{N}$ we have 
\[\frac{1}{3b_n^2}<\frac{\lambda(\{\alpha:a_1(\alpha)=b_1,...,a_n(\alpha)=b_n\})}{\lambda(\{\alpha:a_1(\alpha)= b_1,...,a_{n-1}(\alpha)=b_{n-1}\})}<\frac{2}{b_n^2}.\]
\end{lem}

From this it is an easy exercise to deduce:

\begin{cor}\label{cor:bounds}
\[\frac{1}{3b_n}<\frac{\lambda(\{\alpha:a_1(\alpha)=b_1,...,a_n(\alpha) \geq b_n\})}{\lambda(\{\alpha:a_1(\alpha)= b_1,...,a_{n-1}(\alpha)=b_{n-1}\})}<\frac{4}{b_n}.\]
\end{cor}

Let $W_n=\left\{\alpha: \underset{i=1}{\overset{n}{\sum}}a_i(\alpha)<10 n \log n\right\}$.

\begin{lem}\label{lem:Wn}
$\lambda(W_n)>\frac {1}{10}$ for $n\geq 10$.
\end{lem}

\begin{proof} Let $A_n=\{\alpha:a_i(\alpha)<n^2 \text{ for all }i\leq n\}$. By Corollary \ref{cor:bounds}, $\lambda(\{a_i(\alpha)\geq n^2\})<\frac{4}{n^2}$ for any $i$. Thus, $\lambda(A_n^c)<\frac{4}{n}$.

Consider $\underset{i=1}{\overset{n}{\sum}}\int_{A_n}a_i(\alpha)d\lambda.$ We have,
\begin{align} \underset{i=1}{\overset{n}{\sum}}\int_{A_n}a_i(\alpha)d\lambda & = \sum_{i=1}^n \sum_{j=1}^{n^2} j\cdot\lambda(A_n\cap\{a_i(\alpha)=j\}) \nonumber \\
	& < \sum_{i=1}^n \sum_{j=1}^{n^2} j \frac{2}{j^2} \nonumber
\end{align}
where we have bounded $\lambda(A_n\cap \{a_i(\alpha)=j\})<\frac{2}{j^2}$ using Lemma \ref{bounds}. The double sum is less than or equal to $2n(1+\log {n^2})$ which is bounded above by $5n\log n$, for $n\geq 10$.

Using Markov's inequality, we have
\begin{align}
	\lambda(W_n^c\cap A_n) &\leq \frac{1}{10n \log n} \int_{\alpha\in A_n} \sum_{i=1}^n a_i(\alpha) \ d\lambda \nonumber \\
		& \leq \frac{1}{10n \log n}  5n\log n   \nonumber \\
		& \leq \frac{1}{2}. \nonumber
\end{align}
Since $\lambda(A_n) \geq 1-\frac{4}{n}$, we conclude that $\lambda(W_n\cap A_n)\geq \frac{1}{2}-\frac{4}{n}$. Therefore, $\lambda(W_n) > \frac{1}{10}$ for $n\geq 10$.
\end{proof}

\begin{rem}
The bound in Lemma \ref{lem:Wn} is not optimal, as is easily seen from the proof. We are only concerned to find some bound away from zero.
\end{rem}

We are now ready to prove Proposition \ref{big time}.

\begin{proof}[Proof of Proposition \ref{big time}] 
Fix $C>0$. Corollary \ref{cor:bounds} and the definition of $W_{m-1}$ imply that
\begin{align}
	 \lambda\Big(\{\alpha:  a_1(\alpha)=b_1, \ \ldots \ , & \ a_{m-1}(\alpha)=b_{m-1},  \nonumber \\ 
	 		& \text { and }a_m(\alpha) \geq 10C(m-1) \log(m-1)\} \cap W_{m-1}\Big) \nonumber \\
	 &\geq \frac{\lambda(\{\alpha:  a_1(\alpha)=b_1, \ \ldots, \ a_{m-1}(\alpha)=b_{m-1}\}\cap W_{m-1})}{30C(m-1)\log(m-1)}. \nonumber 
\end{align}
From this we have that
\[ \lambda\Big( W_{m-1}\cap \{ \alpha:a_m(\alpha)\geq 10 C(m-1)\log(m-1)\}\Big) \geq \frac{\lambda(W_{m-1})}{30 C(m-1)\log(m-1)}. \]
Let $G_m=W_{m-1}\cap \{ \alpha:a_m\geq 10C(m-1)\log(m-1)\}.$  Notice that $\alpha \in G_m$ implies that $a_m(\alpha) > C \sum_{i=1}^{m-1} a_i(\alpha)$. Then, for $m\geq 10$, using Lemma \ref{lem:Wn}
\[ \lambda(G_m)\geq \frac{\lambda(W_{m-1})}{30C(m-1)\log(m-1)} > \frac{1}{300C(m-1)\log(m-1)}.\]
Using this estimate,
\[ \sum_{m=1}^\infty \lambda(G_{m}) > \sum_{m=10}^\infty \frac{1}{300C(m-1)\log(m-1)} = \infty. \]

To complete the proof we need two lemmas. The first is a well known partial converse to the Borel-Cantelli Lemma for \emph{quasi}-independent sets (as opposed to independent sets). Its proof is included for completeness.

\begin{lem}
Let $A_i$ be measurable subsets of a space with probability measure $\lambda$. If there exists $C>0$ such that $\lambda(A_i\cap A_j)<C\lambda(A_i)\lambda(A_j)$ and  $\sum_{i=1}^{\infty}\lambda(A_i)=\infty$,  then $\lambda\left( \LS A_i\right)>\frac{1}{4C}>0$.
\end{lem}

\begin{proof} 
Let $B_{N,M}=\cup_{i=N}^M A_i$. If $\sum_{i=N}^M\lambda(A_i)<\frac 1 {2C}$ then for any $j \notin [N,M]$ we have that 
\[\lambda(A_j \setminus B_{N,M})\geq \lambda(A_j)-\sum_{i=N}^M C\lambda(A_j)\lambda(A_i)>\frac{1}{2} \lambda(A_j).\]
Because $\sum \lambda(A_i)=\infty$, the above implies that $\lambda(B_{N,\infty})\geq \frac{1}{4C}$ for all $N$. Because we are in a finite measure space it follows that $\lambda\left(\LS A_i\right)=\lim_{N \to \infty} \lambda\left(\cup_{i=N}^{\infty}A_i)\right)$ and so is at least $\frac{1}{4C}$.
\end{proof}

\begin{lem} 
If $m>n\geq 10$ then $\lambda(G_m\cap G_n)\leq 120 \lambda(G_m)\lambda(G_n)$. 
\end{lem}

\begin{proof}
By Corollary \ref{cor:bounds}, if $A_1=\{\alpha:a_i(\alpha)=b_i \text{ for }1\leq i<m\}, A_2=\{\alpha:a_i(\alpha)=c_i \text{ for }1\leq i<m\}$ are both subsets of $W_{m-1}$ then 
\begin{equation}\label{eqn:ind bounds}
	\frac{1}{12} \frac{\lambda(A_2\cap G_m)}{\lambda(A_2)}\leq \frac{\lambda(A_1\cap G_m)}{\lambda(A_1)}\leq 12 \frac{\lambda(A_2\cap G_m)}{\lambda(A_2)}. 
\end{equation}
Write $G_n = \sqcup_i A_i$ with each $A_i$ of the form $A_i=\{\alpha: a_\ell(\alpha)=b_\ell, 1\leq \ell <m \}.$  
Then $G_m \cap G_n = \sqcup_i (G_m \cap A_i)$, where we can assume all $A_i \subset W_{m-1}$. Then, using equation \ref{eqn:ind bounds},
\begin{align}
	\lambda(G_m \cap G_m) &= \sum_i \lambda(G_m \cap A_i) \nonumber \\
			& \leq \sum_i 12 \frac{\lambda(A^* \cap G_m)}{\lambda(A^*)} \lambda(A_i) \nonumber \\
			& \leq 12 \lambda(G_n) \frac{\lambda(A^* \cap G_m)}{\lambda(A^*)} \nonumber
\end{align}
for an arbitrary $A^* \subset W_{m-1}$ of the form above. Since a subcollection of the $A_i$ form a partition of $W_{m-1}$, by restricting the above estimate to that subcollection we have $\lambda(G_m\cap G_n) \leq 12 \frac{1}{\lambda(W_{m-1})}\lambda(G_m)\lambda(G_n)$. The result follows, using Lemma \ref{lem:Wn}.
\end{proof}

 Applying these two lemmas we conclude that there is a positive measure set of $\alpha$ for which $a_m(\alpha) \geq C \sum_{i=1}^{m-1} a_i(\alpha)$ infinitely often. If $\alpha$ is in this set, its image under the Gauss map is as well, so by the ergodicity of that map the set of such $\alpha$ in fact has full measure.
\end{proof}

The following two lemmas on the shrinking targets $U_j$ are also needed to complete our proof of Theorem B. Recall that $U_j=R_\alpha^j(V_j)$ and 
\[V_j = \{x : x\in C_{c_0, \ldots, c_j} \mbox{ and such that } c_0, \ldots, c_j, 0 \mbox{ and } c_0, \ldots, c_j, 1 \in \Sigma\}.\] 
These lemmas are proved using the partial fraction expansion of $\alpha$.  We will denote by $\{y\}$ the value modulo 1 of a real number $y$ and by $\langle\langle y\rangle\rangle$ the distance from $y$ to the nearest integer.

\begin{lem}\label{atom_description}
Let
\[r_j = max\{q_k:q_k \leq j\}\]
\[s_j = max\{q_k:q_{k+1} \leq j\}\]
\[t_j = max\{T\in \mathbb{N}:s_j+Tr_j \leq j\}.\]
Then
\[ R_\alpha(U_j) = \big[ \{s_j\alpha\}+t_j\{r_j\alpha\}, \{r_j\alpha\}\big)\]
or
\[ R_\alpha(U_j) = \big[ \{r_j\alpha\}, \{s_j\alpha\}-t_j(1-\{r_j\alpha\})\big),\]
and
\[\lambda(U_j)=\lambda(V_j) = \langle\langle r_j\alpha \rangle\rangle + \langle\langle s_j\alpha \rangle\rangle - t_j\langle\langle r_j\alpha \rangle\rangle.\]
\end{lem}

\begin{rem}
Note that if $r_j=q_k$, $s_j=q_{k-1}$ and $t_j<a_{k+1}$.
\end{rem}

\begin{proof}
Note that $\langle\langle r_j\alpha \rangle\rangle$ is smaller than or equal to $\langle\langle i\alpha\rangle\rangle$ for all $i\leq j$.
\vspace{.5cm}

\noindent \textsc{Case 1:}  $0<\{r_j\alpha\}<1/2$.  As the convergents alternate in approximating $\alpha$ from above and below, $1/2 < \{s_j\alpha\}<1$.  The only possible improvement in $\{r_j\alpha\}$ as an upper bound for $R_\alpha(U_j)$ would come from finding some $l$ with $\langle\langle l\alpha \rangle\rangle < \langle\langle r_j\alpha \rangle\rangle$.  This is not possible for $l\leq j$.  Thus the upper endpoint of $R_\alpha(U_j)$ is $\{r_j\alpha\}$ as desired.

The lower bound on $R_\alpha(U_j)$ given by $\{s_j\alpha\}$ can be improved only by adding $\{r_j\alpha\}$ some number of times, as $r_j$ is the only integer $\leq j$ with $\langle\langle r_j\alpha \rangle\rangle < \langle\langle s_j\alpha \rangle\rangle$.  The lower endpoint will thus be of the form $y = \{s_j\alpha\} + T\{r_j\alpha\}$ and will be found by taking $T$ as large as possible such that the $s_j+Tr_j$ rotations required to produce this point do not exceed $j$; this number is $t_j$.

We calculate that $\lambda(U_j) = \langle\langle r_j\alpha \rangle\rangle + (1-\{s_j\alpha\}-t_j\{r_j\alpha\})$ using the fact that in this case $\langle\langle r_j\alpha \rangle\rangle = \{r_j\alpha\}.$  Since $\langle\langle s_j\alpha \rangle\rangle = 1-\{s_j\alpha\}$, this simplifies to the desired result.
\vspace{.5cm}

\noindent \textsc{Case 2:} $1/2<\{r_j\alpha\}<1$.  Then $0<\{s_j\alpha\}<1/2$ and the lower endpoint of $R_\alpha(U_j)$ is $\{r_j\alpha\}$.  As before, the upper endpoint is of the form $\{s_j\alpha\} - T(1-\{r_j\alpha\})$.  The best such endpoint is found by taking $T$ as large as possible, i.e. equal to $t_j$.

Finally, we calculate again
\begin{align}
	\lambda(U_j) =& \langle\langle s_j\alpha \rangle\rangle - t_j(1-\{r_j\alpha\})+(1-\{r_j\alpha\}) 												\nonumber \\
				= & \langle\langle s_j\alpha \rangle\rangle - t_j\langle\langle r_j\alpha \rangle \rangle +\langle\langle r_j\alpha \rangle\rangle . \nonumber
\end{align}

\end{proof}

For use in the lemma below as well as later in the paper, we fix some notation. We will adopt interval notation ($[n,m)$, etc.) to denote intervals of integers; context will make the distinction between these and subsets of the real interval $[0,1)$ clear. 

We let $I_i = [q_i, q_{i+1})$. We let 
\begin{equation*}
	J^i_b = \left\{  
	\begin{array}{ll}
		[q_i, q_{i-1}+q_i) & \mbox{ if } b=1,\\
		\phantom{}[q_{i-1}+(b-1)q_i, q_{i-1}+bq_i) & \mbox{ if } 1< b\leq a_{i+1}.
	\end{array}\right.
\end{equation*}
Let $\mathcal{J}$ denote the collection of all the $J^i_b$'s. We note that $J^i_b\subset I_i$ and that these intervals are disjoint. 

Further, let $J^i_2 = [q_{i-1}+q_i, q_{i-1}+2q_i)$ for all $i$, whether $a_{i+1}\geq 2$ or not. If $a_{i+1}=1$, $J^i_2 \subset I_{i+1}$ and it equals $J^{i+1}_1$, but we note that in any case $\{J^i_2\}_{i\in\mathbb{N}}$ consists of pairwise disjoint intervals.

\begin{lem}\label{disjoint}
For any $J \in \mathcal{J}$, and for all $l\in J$, the sets $V_l = R_\alpha^{-l}U_l$ are pairwise disjoint. 
\end{lem}

\begin{proof}
Fix $J_b^i \in \mathcal{J}$. For $l\in J_b^i$, Lemma \ref{atom_description} tells us that $R_\alpha U_l$ is the interval containing 0 bounded by $R_\alpha^{q_i}(0)$ and $R_\alpha^{q_{i-1}+(b-1)q_i}(0)$.

Suppose that $l>k\in J_b^i$. Then $U_l=U_k=:U$, and $R_\alpha^{-l} U \cap R_\alpha^{-k} U \neq \emptyset$ if and only if $R_\alpha U \cap R_\alpha ^{l-k}(R_\alpha U) \neq \emptyset$. For such an intersection to occur, $R_\alpha ^{l-k}$ of some endpoint of $R_\alpha U$ must lie in $R_\alpha U$.

We examine the two cases: $b=1$ and $b>1$.

If $b=1$, $J_b^i=[q_i, q_{i-1}+q_i)$, $1<l-k<q_{i-1}$, and the endpoints of $R_\alpha U$ are $R_\alpha^{q_i}(0)$ and $R_\alpha^{q_{i-1}}(0)$. The first time after $q_{i-1}$ that the orbit of 0 hits $U$ is $q_{i-1}+q_i$. But $(l-k)+q_{i-1}<q_{i-1}+q_i$ and $(l-k)+q_i<q_{i-1}+q_i$, so neither endpoint of $R_\alpha U$ will return to $R_\alpha U$ under $R_\alpha^{l-k}$, proving the desired disjointness.

If $b>1$, $J_b^i=[q_i+(b-1)q_i, q_{i-1}+bq_i)$, $1<l-k<q_i$, and the endpoints of $R_\alpha U$ are $R_\alpha^{q_i}(0)$ and $R_\alpha^{q_{i-1}+(b-1)q_i}(0)$. The first time after $q_{i-1}+(b-1)q_i$ that the orbit of 0 hits $U$ is $q_{i-1}+bq_i$. But $(l-k)+q_i<q_{i-1}+bq_i$ since $l-k<q_i$ and $b\geq 2$ and $(l-k)+(q_{i-1}+(b-1)q_i) < q_{i-1}+b q_i$ since $l-k<q_i$, so neither endpoint of $R_\alpha U$ will return to $R_\alpha U$ under $R_\alpha^{l-k}$, again proving disjointness.
\end{proof}

\begin{cor}\label{cor:lambda bound}
For all $m$,
\[ \sum_{j=1}^{q_m-1} \lambda(V_j) < \sum_{i=1}^m a_i.\]
\end{cor}

\begin{proof}
In $\mathcal{J}$, there are $a_i$ intervals $J^{i-1}_b$ contained in $I_{i-1} = [q_{i-1},q_i)$. By Lemma \ref{disjoint}, $\sum_{l\in J^{i-1}_b} \lambda(V_l) < 1$ since the $V_l$ are disjoint over these indices. Therefore, $\sum_{j=q_{i-1}}^{q_i-1}\lambda(V_j) < a_i$ and the result follows.
\end{proof}

The following technical tool, a consequence of equidistribution of points under the rotation $R_\alpha$ and regularity of measures will be used in the proof of Theorem B:

\begin{lem}\label{lem:equidist}
Let $A \subset [0,1)$ have positive measure and fix $\delta>0$. Suppose we have families $\{X_m\}_{m\in \mathbb{N}}$ and $\{Y_m\}_{m\in \mathbb{N}}$ of subsets of $[0,1)$ such that
\begin{itemize}
	\item $\lambda(X_m), \lambda(Y_m) > \delta >0$ for all $m$,
	\item For each $m$, $X_m = \bigcup_{k=1}^{K_m} R_\alpha^k(U_m)$ and $Y_m = \bigcup_{k=1}^{K_m} R_\alpha^k(V_m)$ where $U_m$ and $V_m$ are intervals and $K_m\to \infty$ as $m\to \infty$.
\end{itemize}
Them, for any sufficiently large $m$, there exists a pair of points $x^*\in X_m \cap A$ and $y^*\in Y_m \cap A$ with $|x^*-y^*|<\delta$.
\end{lem}

\begin{proof}
Choose a positive $\epsilon$ satisfying $\epsilon< \frac{(.99)\lambda(A)\delta}{2+(.99)\delta}.$ This choice guarantees that $(\frac{1}{2})(.99)(\lambda(A)-\epsilon)\delta > \epsilon$. Since $A$ has finite measure, there is a finite, disjoint union of open intervals $B = \bigsqcup_{i=1}^n I_i$ such that $\lambda(A \Delta B) < \epsilon$. By the equidistribution of points under $R_\alpha$ and the fact that $K_m\to \infty$, we may pick $M>0$ so large that for all $m>M$, 
\[ \lambda(I_i \cap X_m) > .99 \lambda(I_i) \delta\]
\[ \lambda(I_i \cap Y_m) > .99 \lambda(I_i) \delta\]
for all $i=1, \ldots , n$, using our lower bound on the measures of $X_m$ and $Y_m$. Further pick $M$ so large that for $m>M$, the maximum separation between two adjacent points in $\{R_\alpha^k 0\}_{k=1}^{K_m}$ is $<\delta$.

Consider the intervals forming $X_m$ and $Y_m$ which are contained in $I_i$. For each interval $U$ which is a connected component of $X_m$, let $V_U$ be its nearest neighbor to the right among the connected components of $Y_m$. (Such a neighbor exists for all but possibly the last such $U$ contained in $I_i$. We may choose $M$ so large that the number of $X_m$ intervals in $I_i$ is very large, making this exceptional subinterval's contribution to the argument below negligible.) Note that $\max_{x\in U, y\in V_U}|x-y|<\delta$.  If a pair $(x^*, y^*)$ as desired does not exist, then for each pair $(U, V_U)$, at least one of $U, V_U$ contains no points in $A$. Therefore, $\lambda(I_i \setminus A) > \frac{1}{2}(.99)\lambda(I_i)\delta$. Thus,
\begin{align}
	\lambda(B\setminus A) &>  \sum_{i=1}^n \frac{1}{2}(.99) \lambda(I_i)\delta \nonumber \\
					& = \frac{1}{2}(.99)\lambda(B)\delta \nonumber \\
					& > \frac{1}{2}(.99)(\lambda(A)-\epsilon)\delta > \epsilon \nonumber
\end{align}
by our choice of $\epsilon$. But this contradicts our choice of $B$, proving the lemma.
\end{proof}

To simplify notation a bit, we set for all integers $m$:
\[ f_m(x) := \frac{\sum_{j=1}^{q_m-1}\chi_{V_j}(x)}{\sum_{j=1}^{q_m-1} \lambda(V_j)}. \]
Where it exists, we set
\[ f(x) := \lim_{N\to\infty} \frac{\sum_{j=1}^N \chi_{V_j}(x)}{\sum_{j=1}^N \lambda(V_j)}.\]
Note that, where it exists, $\lim_{m\to \infty} f_m(x) = f(x)$ and $f$ is measurable. In addition, by Fatou's Lemma, $f$ will be integrable over the set where it is defined, since $\int_{[0,1)}f_md\lambda =1$ for all $m$. Therefore we can assume $f$ takes only finite values.

We are now ready to prove Theorem B.

\begin{proof}[Proof of Theorem B]
Fix $C>0$ and apply Proposition \ref{big time} to find a full measure set of $\alpha$ satisfying equation \ref{eqn:a condition} for infinitely many $m$. Fix any such $m$.

For all $b\in [2,a_{m}]$, let
\[ W_b = \bigcup_{j \in J^{m-1}_b} V_j.\]
Note that by Lemma \ref{disjoint}, this is a disjoint union, and using Lemma \ref{atom_description}, 
\[ \lambda(W_b) = q_{m-1}\big[ \langle\langle q_{m-2}\alpha \rangle\rangle - (b-2)\langle\langle q_{m-1}\alpha \rangle\rangle\big].\]
In addition, if $x\in W_b$, then it will belong to exactly one $V_j$ with $j\in J^{m-1}_{b'}$ for all $b'\leq b$, and because $V_{j+q_k}\subset V_j$ for $q_k\leq j\leq q_{k+1}-q_k$,
\[\sum_{j=q_{m-1}} ^{q_m-1} \chi_{V_j}(x) \geq b \  \mbox{ for all } \ x\in W_b.\]

Choose any $\rho\in(1/8,1/4)$ in such a way that $\rho a_m\in \mathbb{N}$ (possible since $a_m$ is very large), and let $X_m=W_{\rho a_m}$. We then estimate the measure of $X_m$ below using standard results on the convergents:
\begin{align}
	\lambda(X_m) &= q_{m-1}\left[ \langle\langle q_{m-2}\alpha \rangle\rangle - (\rho a_m-1)\langle\langle q_{m-1}\alpha \rangle\rangle \right]  \nonumber \\
			& \geq q_{m-1} \left[ \frac{1}{q_{m-1}+q_{m-2}} - \rho a_m \frac{1}{q_m} \right] \nonumber \\
			& \geq \frac{1}{2} - \rho \frac{a_m q_{m-1}}{q_m} \nonumber \\
			& \geq \frac{1}{2} - \rho \nonumber \\
			& \geq \frac{1}{4}. \nonumber
\end{align}

Second, choose $\sigma\in (1/16,1/8)$ so that $\sigma a_m \in \mathbb{N}$ and is $\geq 2$. Let $Y_m = W_1 \setminus W_{\sigma a_m}$. Then, as any $y\in Y_m$ will not belong to $V_j$ when $j\in J^m_b$ for $b\geq \sigma a_m$, 
\[ \sum_{j=q_{m-1}}^{q_m-1} \chi_{V_j}(y) \leq \sigma a_m -1\ \mbox{ for all } \ y \in Y. \]
Using Corollary \ref{cor:lambda bound}, 
\[ \sum_{j=1}^{q_m-1} \chi_{V_j}(y) \leq \sigma a_m -1 + \sum_{i=1}^{m-1}a_i \mbox{ for all } y\in Y.\]
We can also estimate the measure of this set (recalling that $a_m$ is very large):
\begin{align}
	\lambda(Y_m) &= q_{m-1}\big[ (\sigma a_m -1)\langle\langle q_{m-1}\alpha \rangle\rangle\big] \nonumber \\
			& \geq q_{m-1} (\sigma a_m-1) \frac{1}{q_m+q_{m-1}} \nonumber \\
			& \geq \frac{1}{2} \sigma \frac{a_m q_{m-1}}{2 q_m} \nonumber \\
			& = \frac{1}{4} \sigma \frac{a_m q_{m-1}}{a_mq_{m-1}+q_{m-2}} \nonumber \\
			& \geq \frac{1}{4} \sigma \frac{a_m q_{m-1}}{(a_m+1)q_{m-1}} \nonumber \\
			& \geq \frac{1}{4} \sigma \frac{1}{2} \  \geq \frac{1}{128}.\nonumber 
\end{align}
Estimates here are certainly not precise; the key point is that $X_m$ and $Y_m$ have a positive lower bound on their measures which is independent of $m$. Let $\delta =\frac{1}{128}$.

Using the results above, for all $x\in X_m$ and $y\in Y_m$,
\begin{align}
	\sum_{j=1}^{q_m-1} \chi_{V_j}(x) - \sum_{j=1}^{q_m-1} \chi_{V_j}(y) & > \rho a_m -\sigma a_m +1 - \sum_{i=1}^{m-1} a_i \nonumber \\
%			& > (\rho-\sigma)a_m - \sum_{i=1}^{m-1} a_m \nonumber \\
			& > (\rho -\sigma - 1/C)a_m. \nonumber
\end{align}
Finally, using Corollary \ref{cor:lambda bound},
\begin{align}
	\big|f_m(x)-f_m(y)\big| = \frac{\Big| \sum_{j=1}^{q_m-1} \chi_{V_j}(x) - \sum_{j=1}^{q_m-1} \chi_{V_j}(y)\Big|}{\sum_{j=1}^{q_m-1} \lambda(V_j)} &\geq \frac{(\rho-\sigma-1/C)a_m}{\sum_{i=1}^{m}a_i} \nonumber \\
			& \geq \frac{(\rho-\sigma-1/C)a_m}{(1+1/C)a_m} \nonumber \\
			& = \frac{(\rho-\sigma-1/C)}{(1+1/C)}. \nonumber
\end{align}
By choosing $C$ sufficiently large, and since $\rho>\sigma$, we have $|f_m(x)-f_m(y)|\geq D>0$ for all $m$ such that equation \ref{eqn:a condition} holds and all $x\in X_m$, $y\in Y_m$.

Let $Z=\{x: f(x) \mbox{ exists}\}$. Towards a contradiction, assume $\lambda(Z)>0$. Fix $\epsilon< \frac{D}{3}$ and $<\frac{\lambda(Z)}{2}$.

Since $f$ is measurable, by Luzin's Theorem there is a compact set $G\subset Z$ with $\lambda(G)>\lambda(Z)-\epsilon$ over which $f$ is (uniformly) continuous. Let $\delta>0$ be such that $|x-y|<\delta$ and $x,y\in G$ imply $|f(x) - f(y)|<\epsilon$.

Let
\[Z_N =\big\{x \in Z: \mbox{ for all } n\geq N, \frac{\sum_{j=1}^n\chi_{V_j}(x)}{\sum_{j=1}^n\lambda(V_j)} \mbox{ is within } \epsilon \mbox{ of } f(x)\big\}.\]
Under our assumption $\lambda(Z_N) \to \lambda(Z)$ as $N\to \infty$. Pick $N_0$ so large that $\lambda(Z_{N_0})>\lambda(Z)-\epsilon$ and, therefore, $\lambda(G\cap Z_{N_0})>\lambda(Z)-2\epsilon >0$ by the choice of $\epsilon$.

Let $m$ be chosen so large that the following hold:

\begin{itemize}
	\item $q_m > N_0$,
	\item $a_m$ satisfies condition \ref{eqn:a condition}, and
	\item $\{X_m\}_{m\in\mathbb{N}}$ and $\{Y_m\}_{m\in\mathbb{N}}$ satisfy Lemma \ref{lem:equidist} with $A=G \cap Z_{N_0}$.
\end{itemize}
Then we may take $x^*\in X_m\cap G\cap Z_{N_0}$ and $y^*\in Y_m\cap G\cap Z_{N_0}$ with $|x^*-y^*|<\delta$. As $x^*, y^*\in Z_{N_0}$ and $q_m>N_0$, $|f_m(x^*)-f(x^*)|<\epsilon$ and $|f_m(y^*)-f(y^*)|<\epsilon$. As both points are in $G$ and $|x^*-y^*|<\delta$, $|f(x^*)-f(y^*)|<\epsilon$. We conclude that $|f_m(x^*) - f_m(y^*)| < 3\epsilon < D$. But this contradicts our result above on the minimum difference between the values of $f_m$ at points in $X_m$ and $Y_m$ when $a_m$ satisfies \ref{eqn:a condition}. Therefore there is a set of full measure where the $f_m$ do not converge, completing the proof.
\end{proof}

%
%%
%%%
%%%%
%%%%%
%%%%%%
%%%%%%%
%%%%%%%%%%%
%%%%%%%%%%%

\section{Proof of Theorem A}\label{sec:main proof}

Towards Theorem A, we claim the following set of inequalities:

There exists a positive constant $C_1$ such that for almost every $\alpha$ and $x\in [0,1)$,
\begin{equation}\label{pointwise_bounds}
	C_1n(\log n)^3 > \sum_{i=1}^n a_i(\alpha) \geq \sum_{j=1}^{q_n-1}\chi_{V_j}(x) > \frac{1}{4}(n-2).
\end{equation}
The middle inequality follows from almost the same proof as Corollary \ref{cor:lambda bound}. We prove the other two inequalities in the following sequence of Lemmas. Lemma \ref{alpha_bound} specifies the full measure set of $\alpha$ for which we prove Theorem A.

\begin{lem}\label{alpha_bound}
There exists a positive constant $C_1$ such that for almost every $\alpha$, $C_1n(\log n)^3 > \sum_{i=1}^n a_i(\alpha)$ for all sufficiently large $n$.
\end{lem}

\begin{rem}
Note that how large $n$ must be for the given bound to hold does depend on $n$.
\end{rem}

\begin{proof}
As in the proof of Lemma \ref{lem:Wn}, set $A_n=\{\alpha: a_i(\alpha) < n^2 \mbox{ for all }  i \leq n\}$. As before, $\int_{A_n}\underset{i=1}{\overset{n}{\sum}}a_i(\alpha)d\lambda(\alpha)\leq 5n \log n$ (for $n>7$).
Note also that $\lambda$-a.e. $\alpha$ belongs to $A_n$ for all but finitely many $n$.
It follows from Markov's inequality that 
\begin{align}
	\lambda\left(\{\alpha\in A_n: \underset{i=1}{\overset{n}{\sum}}a_i(\alpha) > 10 n (\log n)^{2.1}\}\right) &\leq \frac{1}{10n(\log n)^{2.1}}\int_{A_n} \sum_{i=1}^n a_i(\alpha) d\lambda \nonumber \\
		& \leq \frac{1}{2}\Big( \frac{1}{\log n}\Big)^{1.1}. \nonumber
\end{align}

Since almost every $\alpha$ belongs to $A_n$ for all but finitely many $n$, almost every $\alpha$ belongs to $A_{10^k}$ for all but finitely many $k$. Then
\[ \lambda\left(\left\{\alpha\in A_{10^k}: \underset{i=1}{\overset{10^k}{\sum}}a_i(\alpha) > 10^{k+1} (\log 10^k)^{2.1}\right\}\right) \leq \left(\frac{1}{\log 10^k}\right)^{1.1}.\]
These measures form a summable sequence, so for a.e. $\alpha$, 
\[ \sum_{i=1}^{10^k} a_i(\alpha) \leq 10^{k+1}\left(\log 10^k\right)^{2.1} \mbox{ for all but finitely many } k.\]
This implies the Lemma for the subsequence $n=10^k$ because for large enough $k$, we have $10^{k}(\log 10^{k})^3 > 10^{k+1} (\log 10^{k+1})^{2.1}.$

In general, given $\alpha$, suppose that $\sum_{i=1}^{10^k}a_i(\alpha) < 10^k (\log 10^k)^3$ for all $k\geq k^*$. Then for any $n\geq 10^{k^*}$, if $10^k\leq n\leq 10^{k+1}$, then
\[ \sum_{i=1}^n a_i(\alpha) \leq 10^{k+1}(\log 10^{k+1})^3 \leq 10 n (\log 10 n)^3\]
and the result holds after an appropriate choice of $C_1$.
\end{proof}

We will give a lower bound on $\sum_{j=1}^{q_n}\chi_{V_j}(x)$ by bounding below the sum over the $J_2^i$.  As we noted above, $\{J_2^i\}_{i\in\mathbb{N}}$ is a disjoint set of intervals. Let 
\[h_i(x) = \sum_{j\in J_2^i}\chi_{V_j}(x).\]

\begin{lem}\label{one_half}
For all $i$,
\[\int_{[0,1)} h_i(x)d\lambda >1/2.\]
\end{lem}

\begin{proof}
As per Lemma \ref{disjoint}, over $j\in J_2^i$, the $V_j$ are disjoint, so $h_i(x) \in \{0, 1\}$.  The length of the interval $J_2^i$ is $q_i$, and for $j\in J^i_2$,
\[\lambda(V_j) = \langle\langle q_{i-1}\alpha \rangle\rangle,\]
using the description of $R_\alpha(U_j)$ provided by Lemma \ref{atom_description}.  By Theorem 13 in \cite{khinchin}, $\langle\langle q_{i-1}\alpha \rangle\rangle > \frac{1}{q_{i-1}+q_i}$.  We may then bound the integral from below by 

\[\int_{[0,1)} h_i(x)d\lambda > \frac{q_i}{q_i+q_{i-1}} > \frac{q_i}{2q_i} = \frac{1}{2}.\]
\end{proof}

The following sequence of results prove that the random variables $h_i(x)$ are (approximately) independent.

\begin{lem}
Let $[c,d) \subset [0,1)$.  Let $f_{[c,d)}(i,b) = \#\{[c,d) \cap \cup_{l \in J^i_b} R_\alpha^{-l}(0)\}$.  Then 
\[\lambda\left([c,d)\right)\left|J^i_b\right|-2 \leq f_{[c,d)}(i,b) \leq \lambda\left([c,d)\right)\left|J^i_b\right| +2.\]
\end{lem}

\begin{proof}
By Theorem 1 in \cite{kesten}, each interval $(\frac{j}{q_m}, \frac{j+1}{q_m})$ for $j=0,1, \ldots, q_m-1$ contains exactly one point of $R_\alpha^{-l}(0)$ with $1\leq l \leq q_m$.  Recall that $|J_b^i|=q_m$ where $m=i-1$ if $b=1$ and $m=i$ if $b>1$. Therefore, if $a=\min J^i_b$, each $I_j:=R_\alpha^{-(a-1)}(\frac{j}{q_m},\frac{j+1}{q_m})$ contains exactly one point of $R^{-l}(0)$ for $l\in J^i_b$.

  At least $\lambda([c,d))|J_b^i|-2$ of the $I_j$ above are completely contained in $[c,d)$, and at most $\lambda([c,d))|J_b^i|+2$ of them intersect $[c,d)$. The result then follows.
\end{proof}

\begin{prop}\label{prop:independent}
Fix $k$. For all $i$ such that $q_i>k$ and any $1\leq b \leq a_{i+1}$,
\begin{align}
	\left(\frac{\lambda(V_k)|J^i_b|-3}{\lambda(V_k)|J^i_b|}\right) \lambda(V_k) &\lambda\left(\cup_{l\in J^i_b}V_l\right) \nonumber \\
	& \leq  \lambda\left(V_k \cap \bigcup_{l\in J^i_b}V_l\right) \nonumber \\
	& \leq \left(\frac{\lambda(V_k)|J^i_b|+3}{\lambda(V_k)|J^i_b|}\right) \lambda(V_k)\lambda\left(\cup_{l\in J^i_b}V_l\right).\nonumber
\end{align}
\end{prop}

\begin{proof}
Fix $k$. Let $i$ be so large that $q_i>k$. By the previous lemma, the interval $V_k$ is hit by the left endpoints of the $V_l$ between $\lambda(V_k)|J^i_b|-2$ and $\lambda(V_k)|J^i_b|+2$ times.  As the sets $V_l$ are disjoint and of the same measure over $l\in J^i_b$, this easily yields
\[\left(\lambda(V_k)|J^i_b|-3\right) \lambda(V_{l^*}) \leq \lambda \left(V_k \cap \bigcup_{l\in J^i_b}V_l\right) \leq \left(\lambda(V_k)|J^i_b|+3\right) \lambda(V_{l^*}) \ \mbox{ for any } l^* \in J_b^i. \]
Furthermore, for any $l^*\in J^i_b$, $|J^i_b|\lambda(V_{l^*}) = \lambda(\cup_{l\in J^i_b}V_l)$. Translating to an inequality with multiplicative errors yields
\begin{align}
	\left(\frac{\lambda(V_k)|J^i_b|-3}{\lambda(V_k)|J^i_b|}\right) \lambda(V_k) &\lambda\left(\cup_{l\in J^i_b}V_l\right) \nonumber \\
	& \leq  \lambda\left(V_k \cap \bigcup_{l\in J^i_b}V_l\right) \nonumber \\
	& \leq \left(\frac{\lambda(V_k)|J^i_b|+3}{\lambda(V_k)|J^i_b|}\right) \lambda(V_k)\lambda\left(\cup_{l\in J^i_b}V_l\right).\nonumber
\end{align}

\end{proof}

Proposition \ref{prop:independent} asserts near independence of the events $V_k$ and $\cup_{l\in J^i_b} V_l$. Using it for all $k\in J^j_{b'}$ where $j<i$ (which guarantees $k<q_i$) we get the following corollary. It relates to calculating the correlation between a point being undetermined in the intervals $J^j_{b'}$ and $J^i_b$.

\begin{cor} \label{bound1}
For any $k\in J^i_{b'}$, and $J^i_{b'}, J^j_b$ disjoint, $j>i$,
\begin{align}
	\left(\frac{\lambda(V_k)|J^j_b|-3}{\lambda(V_k)|J^j_b|}\right) \lambda\left(\cup_{k\in J^i_{b'}}V_k\right) &\lambda\left(\cup_{l\in J^j_b}V_l\right) \nonumber \\
	& \leq  \lambda\left(\bigcup_{k \in J^i_{b'}}V_k \cap \bigcup_{l\in J^j_b}V_l\right) \nonumber \\
	& \leq \left(\frac{\lambda(V_k)|J^j_b|+3}{\lambda(V_k)|J^j_b|}\right) \lambda\left(\cup_{k\in J^i_{b'}}V_k\right)\lambda\left(\cup_{l\in J^j_b}V_l\right).\nonumber
\end{align}
\end{cor}

\begin{proof}
This follows from summing Proposition \ref{prop:independent}'s inequalities over the disjoint sets $V_k$ for $k\in J^i_{b'}$.  (The desire to compute this sum explains our preference for the formulation in terms of multiplicative bounds above.)
\end{proof}

\begin{prop}\label{integral_bounds} 
For $j>i$
\[\left( 1 - \frac{3q_{i+1}}{q_j} \right) \int h_i d\lambda \int h_j d\lambda \leq \int h_i h_j d\lambda \leq \left( 1 + \frac{3q_{i+1}}{q_j} \right) \int h_i d\lambda \int h_j d\lambda.\]
\end{prop}

\begin{proof}
First, 
\[\int h_i(x)h_j(x) d\lambda = \int \left(\sum_{l\in J^i_2} \chi_{V_l}(x)  \right)  \left( \sum_{l\in J^j_2} \chi_{V_l}(x) \right)d\lambda.\]
As over $J^i_2$ and over $J^j_2$ the sets $V_l$ are disjoint, the integrand of the above has value 0 or 1 according to whether $x\in \left(\cup_{l \in J^i_2}V_l\right) \cap \left(\cup_{l \in J^j_2}V_l\right)$.  Thus,
\[\int h_i h_j d\lambda = \lambda\left( \bigcup_{l\in J^i_2} V_l \cap \bigcup_{l \in J^j_2}V_l \right).\]

By Corollary \ref{bound1}, for $l\in J_2^i$ we get
\begin{align}
	\left(\frac{\lambda(V_l)|J^j_2|-3}{\lambda(V_l)|J^j_2|}\right) \lambda\left(\cup_{l\in J^i_2}V_l\right) &\lambda\left(\cup_{l\in J^j_2}V_l\right) \nonumber \\
		& \leq  \lambda\left(\bigcup_{l \in J^i_2}V_l \cap \bigcup_{l\in J^j_2}V_l\right) \nonumber \\
		& \leq \left(\frac{\lambda(V_l)|J^j_2|+3}{\lambda(V_l)|J^j_2|}\right) \lambda\left(\cup_{l\in J^i_2}V_l\right)\lambda\left(\cup_{l\in J^j_2}V_l\right).\nonumber
\end{align}

To assess the value of the terms $\left(1\pm \frac{3}{\lambda(V_l)|J^j_2|}\right)$ consider an arbitrary $l\in J^i_2$. As $U_l=R_\alpha V_l$, using the description of $R_\alpha U_l$ given by Proposition \ref{atom_description} and \cite[Theorem 13]{khinchin}, $\lambda(V_l) > \langle\langle q_i\alpha \rangle \rangle > \frac{1}{q_{i-1}+q_i}\geq \frac{1}{q_{i+1}}.$    From its description, $|J^j_2|=q_j$.  Using these two bounds, $\frac{3}{\lambda(V_l)|J_2^j|}<\frac{3 q_{i+1}}{q_j}$.

Returning to our inequalities for $\int h_i h_j$, as the $V_l$ are disjoint over $J^j_2$ or $J^i_2$ we can translate back into integrals as so:
\begin{align}
	\left(1- \frac{3q_{i+1}}{q_j}\right) \int \sum_{l\in J^i_2}\chi_{V_l}(x)d\lambda & \int \sum_{l\in J^j_2}\chi_{V_l}(x)d\lambda \nonumber \\
	&\leq \int h_ih_j d\lambda  \leq \nonumber\\
	 &\left(1+ \frac{3q_{i+1}}{q_j}\right) \int \sum_{l\in J^i_2}\chi_{V_l}(x)d\lambda \int \sum_{l\in J^j_2}\chi_{V_l}(x)d\lambda. \nonumber
\end{align}
These are the desired bounds on $\int h_ih_j d\lambda$.
\end{proof}

The independence result we want is the following.

\begin{prop}\label{decay2}
There exist constants $C, b>0$ such that 
\[\left| \int_{[0,1)} h_i(x)h_j(x)d\lambda - \int_{[0,1)} h_i(x)d\lambda\int_{[0,1)}h_j(x)d\lambda \right| < Ce^{-b|i-j|}.\]
\end{prop}

\begin{proof}
We may assume $j>i$. Using Proposition \ref{integral_bounds}, we need to show that the expression
\[\frac{3q_{i+1}}{q_j}\int h_id\lambda \int h_j d\lambda\]
decays exponentially in $|i-j|$.  A clear upper bound on each of $\int h_i d\lambda, \int h_j d\lambda$ is 1. As $q_{k+2}>2q_k$, $\frac{q_{i+1}}{q_j}$ decays exponentially in $|i-j|$, as desired.
\end{proof}

We can apply this approximate independence to prove the remaining inequality in equation \ref{pointwise_bounds}.  Let $\tilde h_i(x) = h_i(x) - \int h_i(x)d\lambda$, and note that $\tilde h_i(x) \in (-1, 1).$  Let $\tilde s_n(x) = \sum_{i=1}^n \tilde h_i(x)$.

\begin{prop}
For almost every $x \in S^1$, for sufficiently large $n$,
\[\sum_{j=1}^{q_n-1}\chi_{V_j} (x) > \frac{1}{4}(n-2). \]
\end{prop}

\begin{proof}
First, for all $x\in [0,1)$, $\sum_{j=1}^{q_n-1}\chi_{V_j} (x) \geq \sum_{i=1}^{n-2} h_i(x)$ as $j\in J^i_2$ implies $j<q_{i+2}$.

Consider $\sum_{i=1}^{n-2} \int h_i(x)d\lambda$.  By Lemma \ref{one_half} this is bounded below by $\frac{1}{2}(n-2)$; it is bounded above by $n$ as $h_i$ takes only 1 or 0 as a value.  Applying Chebyshev's inequality to $\tilde s_n$ yields (for any $\epsilon >0$)
\begin{align*}
\lambda\left(\{x: |\tilde s_{n-2}(x) | > \epsilon (n-2)\}\right) & < \frac{\int \tilde s_{n-2}^2(x) d\lambda}{\epsilon^2 (n-2)^2} \\
								& = \frac{\sum_{i=1}^{n-2} \int \tilde h_i^2(x) d\lambda + 2\sum_{i<j} \int \tilde h_i(x) \tilde h_j(x) d\lambda }{\epsilon^2(n-2)^2}\\
								& < \frac{D}{\epsilon^2 (n-2)}.
\end{align*}
For the last inequality we have used the facts that $\tilde h_i(x) \in (-1,1)$ and therefore $\sum_{i=1}^{n-2} \int \tilde h_i^2(x)d\lambda < n-2$, and that for some positive constant $D$, $2\sum_{i<j} \int \tilde h_i \tilde h_j d\lambda <(D-1)(n-2)$ by Proposition \ref{decay2}.

We restrict our attention to the subsequence of times $\{(n-2)^2\}$, obtaining
\[\lambda(\{x: |\tilde s_{(n-2)^2}(x) | > \epsilon (n-2)^2\}) < \frac{D}{\epsilon^2 (n-2)^2}.\]
Summing the term on the right-hand side of the above inequality over all $n$ yields a convergent series so by the Borel-Cantelli Lemma, for almost every $x\in [0,1)$, 
\[\frac{\tilde s_{(n-2)^2}(x)}{(n-2)^2} \to 0  \quad \mbox{ as } n \to \infty.\]

Consider now the intervals $[(n-2)^2, (n-1)^2)$.  As $\tilde h_i(x) \in (-1, 1)$, for $k\in [(n-2)^2, (n-1)^2)$, 
\[ |\tilde s_{(n-2)^2}(x) - \tilde s_k(x)| < 2(n-2)+1\]
so
\[\frac{|\tilde s_k(x)|}{k} < \frac{|\tilde s_{(n-2)^2}(x)| + 2(n-2)+1}{k} \leq \frac{|\tilde s_{(n-2)^2}(x)| +2(n-2)+1}{(n-2)^2} \to 0\]
as $k\to \infty$. 

We have now that for almost all $x$,
\[\frac{\sum_{i=1}^{n-2} h_i(x) - \int h_i(x)d\lambda }{n-2} \to 0. \]
As $\sum_{i=1}^{n-2} \int h_i(x)d\lambda \in (\frac{1}{2}(n-2), (n-2))$, for sufficiently large $n$, $\sum_{i=1}^{n-2} h_i(x) >\frac{1}{4}(n-2)$ as desired.
\end{proof}

We now prove a similar series of inequalities for $\sum_{j=1}^{q_n} \lambda(V_j)$, namely:

\begin{equation}\label{setwise_bounds}
  C_1n(\log n)^3 > \sum_{i=1}^na_i(\alpha) > \sum_{j=1}^{q_n-1} \lambda(V_j) > \frac{1}{2}(n-2).  
\end{equation}

The left-most inequality is Lemma \ref{alpha_bound} and the next is Corollary \ref{cor:lambda bound}. It remains only to prove:

\begin{lem}For all $\alpha$, 
\[\sum_{j=1}^{q_n-1} \lambda(V_j) > \frac{1}{2}(n-2).\]
\end{lem}

\begin{proof}
This follows easily from Lemma \ref{one_half} after noting that
\[\sum_{j=1}^{q_n-1} \lambda(V_j) > \sum_{i=1}^{n-2}\sum_{j\in J_2^i}\lambda(V_j) = \sum_{i=1}^{n-2}\int_{[0,1)} h_i(x)d\lambda.\]
\end{proof}

The inequalities collected above enable us to prove the main theorem:

\begin{proof}[Proof of Theorem A]
Consider the full measure set of $\alpha$ satisfying Lemma \ref{alpha_bound}. For a given $\alpha$, suppose $n\in [q_m, q_{m+1})$ is large enough for the bound in Lemma \ref{alpha_bound} to hold.  Then we have the following for almost every $x$:
\[ \frac{1}{4}(m-2)< \sum_{j=1}^{q_m-1} \chi_{V_j}(x) \leq \sum_{j=1}^{n} \chi_{V_j}(x) \leq \sum_{j=1}^{q_{m+1}-1} \chi_{V_j}(x) < C_1(m+1)(\log(m+1))^3,\]
and
\[ \frac{1}{2}(m-2)< \sum_{j=1}^{q_m-1} \lambda(V_j) \leq \sum_{j=1}^{n} \lambda(V_j) \leq \sum_{j=1}^{q_{m+1}-1} \lambda(V_j) < C_1(m+1)(\log(m+1))^3.\]

Taking logs and forming the relevant quotient, we see that the $\log(m-2)$ and $\log(m+1)$ terms dominate the $\log(constant)$ and $\log(\log(-))$ terms.  As $\frac{\log(m)}{\log(m+1)}$ and $\frac{\log(m+1)}{\log(m-2)} \to 1$, the result follows.
\end{proof}

%%%%%%%%%%%%%%%%%%%%%%%%%%%%%%%%%%%%%%%%%%%%%%%%%%%%%%%%%%%%%
%%%%%%%%%%%%%%%%%%%%%%%%%%%%%%%%%%%%%%%%%%%%%%%%%%%%%%%%%%%%%

%For acknowledgements section, please don't number the section, please begin it with \section*{Acknowledgements}
\section*{Acknowledgments} We would like to thank unnamed referees for a number of helpful comments.

%%%%%%%%%%%%%%%%%%%%%%%%%%%%%%%%%%%%%%%%%%%%%%%%%%%%%%%%%%%%%
%%%%%%%%%%%%%%%%%%%%%%%%%%%%%%%%%%%%%%%%%%%%%%%%%%%%%%%%%%%%%

\bigskip

%\medskip
% The data information below will be filled by AIMS editorial staff
%Received xxxx 20xx; revised xxxx 20xx.
%\medskip

\end{document}